\input ./style/arxiv-vmsta.cfg
\documentclass[numbers,compress,v1.0.1]{vmsta}
\usepackage{vtexbibtags}

\usepackage{authorquery}

\volume{5}% Updated by VTEXPTS2LaTeX.exe, 09.03.2018 11:10
\issue{1}% Updated by VTEXPTS2LaTeX.exe, 09.03.2018 11:10
\pubyear{2018}
\firstpage{99}% Updated by VTEXPTS2LaTeX.exe, 09.03.2018 11:10
\lastpage{111}% Updated by VTEXPTS2LaTeX.exe, 09.03.2018 11:10
\aid{VMSTA97}
\doi{10.15559/18-VMSTA97}% Updated by VTEXPTS2LaTeX.exe, 07.02.2018
\articletype{research-article}

\startlocaldefs
\def\@journal@url{http://www.vmsta.org}
\def\@credit{%
  \vbox to 0pt{%
    \vskip-1.85pc
    \hskip-\textwidth
    \noindent
    \raise4mm\hbox to \textwidth{%
      \footnotesize
      \href{\@journal@url}{www.vmsta.org}%
      \hfill
      \href{\@vtex@url}{\includevtexlogo}%
      }%
    }%
  }

\newcommand{\rrvert}{\vert}
\newcommand{\llvert}{\vert}

\allowdisplaybreaks

\newtheorem{thm}{Theorem}
\newtheorem{lemma}{Lemma}

\newtheorem{prop}{Proposition}

\theoremstyle{definition}
\newtheorem{defin}{Definition}
\newtheorem{remark}{Remark}

\newcolumntype{d}[1]{D{.}{.}{#1}}
\ifdefined\HCode % tex4ht
\def\newqedhere{}
\def\figend{\end{figure}\begin{figure}}
\else % LaTeX
\def\newqedhere{\qedhere}
\def\figend{\hfill}
\fi
\endlocaldefs
\begin{aqf}
\querytext{Q1}{Please check this sentence. $\varOmega$ is mentioned in
the proof of Theorem 1, $\varOmega'$ is in Proposition 1.}
\querytext{Q2}{Please check this formula.}
\end{aqf}
\begin{document}

\begin{frontmatter}
\pretitle{Research Article}

\title{Fractional Cox--Ingersoll--Ross process with non-zero~<<mean>>}

\author{\inits{Yu.}\fnms{Yuliya}~\snm{Mishura}\ead[label=e1]{myus@univ.kiev.ua}}
\author{\inits{A.}\fnms{Anton}~\snm{Yurchenko-Tytarenko}\thanksref
{cor1}\ead[label=e2]{antonyurty@gmail.com}}
\thankstext[type=corresp,id=cor1]{Corresponding author.}
\address{Faculty of Mechanics and Mathematics, \institution{Taras
Shevchenko National University of Kyiv}, Volodymyrska St., 64/13, Kyiv
01601, \cny{Ukraine}}

%\thankstext[id=f1]{}

%\dedicated{}

%\markboth{Authors}{Title}
\markboth{Yu. Mishura, A. Yurchenko-Tytarenko}{Fractional
Cox--Ingersoll--Ross process with non-zero <<mean>>}

\begin{abstract}
In this paper we define the fractional Cox--Ingersoll--Ross process
as $X_t:=\break Y_t^2\mathbf{1}_{\{t<\inf\{s>0:Y_s=0\}\}}$, where the process
$Y=\{Y_t,t\ge0\}$ satisfies the SDE of the form
$dY_t=\frac{1}{2}(\frac{k}{Y_t}-aY_t)dt+\frac{\sigma}{2}dB_t^H$, $\{B^H_t,
t\ge0\}$ is a fractional Brownian motion with an arbitrary Hurst parameter
$H\in(0,1)$. We prove that $X_t$ satisfies the stochastic differential
equation of the form $dX_t=(k-aX_t)dt+\sigma\sqrt{X_t}\circ dB_t^H$, where
the integral with respect to fractional Brownian motion is considered as
the pathwise Stratonovich integral. We also show that for $k>0$, $H>1/2$
the process is strictly positive and never hits zero, so that actually
$X_t=Y_t^2$. Finally, we prove that in the case of $H<1/2$ the probability
of not hitting zero on any fixed finite interval by the fractional
Cox--Ingersoll--Ross process tends to 1 as $k\rightarrow\infty$.
\end{abstract}

\received{\sday{26} \smonth{9} \syear{2017}}% Updated by
%VTEXPTS2LaTeX.exe, 07.02.2018 08:42
\revised{\sday{29} \smonth{1} \syear{2018}}% Updated by
%VTEXPTS2LaTeX.exe, 07.02.2018 08:42
\accepted{\sday{30} \smonth{1} \syear{2018}}% Updated by
%VTEXPTS2LaTeX.exe, 07.02.2018 08:42
\publishedonline{\sday{5} \smonth{3} \syear{2018}}
\end{frontmatter}

\section{Introduction}\label{sec1}

The classical Cox--Ingersoll--Ross (CIR) process, which was
proposed and studied by Cox, Ingersoll and Ross in \cite{CIR1,CIR2,CIR3}, is the process $r=\{r_t, t\geq0\}$ that satisfies the following stochastic differential equation:
\begin{equation}
\label{standard sde} %
\begin{gathered}
dr_t=(k-ar_t)dt+\sigma\sqrt{r_t}dW_t,\quad a,k,\sigma>0.
\end{gathered}
\end{equation}

Here $a$ corresponds to the speed of adjustment, $k/a$ is called ``the
mean'', $\sigma$ is ``the volatility'', $W=\{W_t,t\ge0\}$~is a Wiener
process and $r_0>0$.

The CIR process is widely used for short-term interest rate modeling as
well as for stochastic volatility modeling in the Heston model \cite
{Hest}. Therefore, in most cases it is also assumed that $2k\ge\sigma
^2$ as, according to \cite{Fel}, if this condition holds, the process
is strictly positive and never hits zero.

It is well known that the CIR process is ergodic and has a stationary
distribution. Moreover, the distribution of its future values $r_{t+T}$
provided that $r_t$ is known is a noncentral chi-square distribution,
and the distribution of the limit value $r_{\infty}$ is a gamma distribution.

However, the real financial models are often characterized by the
so-called ``memory phenomenon'' (see \cite{AI,BM,DGE,YMH} for more
detail), while the standard Cox--Ingersoll--Ross process does not display
it. Therefore, for a better simulation of interest rates or stochastic
volatility it is reasonable to consider a fractional generalization of
the Cox--Ingersoll--Ross process. It should be noted that there are
several approaches to definition of the fractional Cox--Ingersoll--Ross
process. In \cite{LMS1,LMS2} the fractional CIR process is
introduced as a time changed CIR process with inverse stable
subordinator, the so-called ``rough path approach'' is described in
\cite{Mar}. Another way of defining the considered process is presented
in \cite{ER} as part of the discussion on rough Heston models.

The definition of the fractional CIR process with $k=0$, based on the
pathwise integration with respect to fractional Brownian motion, was
also presented in \cite{MMS} for the case $H>2/3$. In \cite{MPRY} it
was shown that for such definition the fractional CIR process is the
square of the fractional Ornstein--Uhlenbeck process until the
first zero hitting (for definition and properties of the fractional
Ornstein--Uhlenbeck process see \cite{Cher}). Based on that, the
fractional CIR process with $k=0$ was defined as the square of the
fractional Ornstein--Uhlenbeck process before its first zero
hitting. It was also shown that such process satisfies the stochastic
differential equation of the form
\begin{equation}
\begin{gathered} dX_t=aX_tdt+\sigma
\sqrt{X_t}\circ dB_t^H, \quad t\ge0,
\end{gathered} %
\end{equation}
where $X_0>0, a\in\mathbb{R},\sigma>0, H\in(0,1)$ and integral with
respect to the fractional Brownian motion is the pathwise Stratonovich
integral. However, due to positive probability of hitting zero, this
process is not suitable for interest rate modeling.

In this paper we introduce a natural generalization of the above model.
First, we consider the process $Y=\{Y_t,t\ge0\}$ which satisfies the
SDE of the form
\begin{equation}
\begin{gathered}
dY_t=\frac{1}{2} \biggl(\frac{k}{Y_t}-aY_t \biggr)dt+\frac{\sigma}{2}dB_t^H, \quad Y_0>0.
\end{gathered} %
\end{equation}
Then, we define the fractional Cox--Ingersoll--Ross process as
the square of $Y_t$ until the first zero hitting moment and show that
it satisfies the SDE of the form
\begin{equation}
\begin{gathered} dX_t=(k-aX_t)dt+\sigma
\sqrt{X_t}\circ dB_t^H, \quad t\ge0,
\end{gathered} %
\end{equation}
where $X_0=Y^2_0>0$ and the integral with respect to the fractional
Brownian motion is defined as the pathwise Stratonovich integral. We
also show that for any $k>0$ and for any Hurst parameter $H>1/2$ the
process is strictly positive and never hits zero.

Next, for the case of $H<1/2$, we prove that the probability that the
fractional CIR process does not hit zero on any fixed finite interval
tends to 1 as $k\rightarrow\infty$. As an auxiliary result, we prove
the analogue of the comparison theorem.

The paper is organized as follows. In Section~\ref{sec2} we define the
fractional CIR process and show that it satisfies the SDE \eqref{eq:
fCIR for intro} with the pathwise Stratonovich integral. In Section~\ref{sec3} we prove that the fractional CIR process is strictly positive
for $k>0$ and $H>1/2$. In Section~\ref{sec4} we prove the analogue of
the comparison theorem and show that in the case of $H<1/2$ the
probability that the fractional CIR process does not hit zero on any
fixed finite interval tends to 1 as $k\rightarrow\infty$. In Appendix
there are simulations that illustrate the results of the paper.

\section{Definition of the fractional Cox--Ingersoll--Ross
process}\label{sec2}

Consider the process $Y=\{Y_t,t\ge0\}$ that satisfies the following SDE
until its first zero hitting:
\begin{equation}
\label{new sde} %
\begin{gathered} dY_t=\frac{1}{2}
\biggl(\frac{k}{Y_t}-aY_t \biggr)dt+\frac{\sigma
}{2}dB_t^H,
\quad Y_0>0, \end{gathered} %
\end{equation}
where $a,k\in\mathbb{R}, \sigma>0$ and $\{B^H_t,t\ge0\}$ is a
fractional Brownian motion with the Hurst parameter $H\in(0,1)$.

\begin{defin}\label{generfCIR} Let $H\in(0,1)$ be an arbitrary Hurst
index, $\{Y_t, t\ge0\}$ be the process that satisfies the equation
\eqref{new sde} and $\tau$ be the first moment of reaching zero by
the latter.
The \emph{fractional Cox--Ingersoll--Ross} process is the
process $\{X_t, t\ge0\}$ such that for all $t\ge0, \omega\in\varOmega$:
\begin{equation}
\label{def: gfCIR} X_t(\omega) = Y_t^2(\omega)
\mathbf{1}_{\{t<\tau(\omega)\}}.
\end{equation}
\end{defin}

Before moving to the main result of this section, let us give the
definition of the pathwise Stratonovich integral.
\begin{defin}
Let $\{X_t, t\ge0\}$, $\{Y_t, t\ge0\}$ be random processes.
The pathwise \emph{Stratonovich integral} $\int_0^TX_s\circ dY_s$ is a
pathwise limit of the following sums
\begin{equation*}
\sum_{k=1}^{n} \frac{X_{t_{k}} + X_{t_{k-1}}}{2}
(Y_{t_{k}} - Y_{t_{k-1}} ),
\end{equation*}
as the mesh of the partition
$0=t_0<t_1<t_2<\cdots<t_{n-1}<t_n=T$
tends to zero, in case if this limit exists.
\end{defin}

\begin{thm}
Let $\tau:=\inf\{s>0:Y_s=0\}$. For $0\le t \le\tau$ the fractional CIR
process from Definition~\ref{generfCIR} satisfies the following SDE:
\begin{equation}
\begin{gathered}\label{eq: fCIR for intro} dX_t=(k-aX_t)dt+
\sigma\sqrt{X_t}\circ dB_t^H, \end{gathered}
\end{equation}
where $X_0=Y^2_0>0$ and the integral with respect to the fractional
Brownian motion is defined as the pathwise Stratonovich integral.
\end{thm}

\begin{proof}
Let us fix an $\omega\in\varOmega$ and consider an arbitrary $t<\tau
(\omega)$.

According to \eqref{new sde} and \eqref{def: gfCIR},
\begin{equation}
\label{eq: 1} X_t = Y_t^2 = \Biggl(
\sqrt{X_0} + \frac{1}{2}\int_0^t
\biggl(\frac
{k}{Y_s}-aY_s \biggr)ds+\frac{\sigma}{2}B_t^H
\Biggr)^2.
\end{equation}

Consider an arbitrary partition of the interval $[0,t]$:
\[
0=t_0<t_1<t_2<\cdots<t_{n-1}<t_n=t.
\]

Using \eqref{eq: 1}, we get
\begin{align*}
X_t &= \sum_{i=1}^{n}
(X_{t_i} - X_{t_{i-1}} ) + X_0
\\
&= \sum_{i=1}^n \Biggl( \Biggl[
\sqrt{X_0} + \frac{1}{2}\int_0^{t_i}
\biggl(\frac{k}{Y_s}-aY_s \biggr)ds+\frac{\sigma}{2}B_{t_i}^H
\Biggr]^2
\\
&\quad- \Biggl[\sqrt{X_0} + \frac{1}{2}\int
_0^{t_{i-1}} \biggl(\frac
{k}{Y_s}-aY_s
\biggr)ds+\frac{\sigma}{2}B_{t_{i-1}}^H \Biggr]^2
\Biggr) + \xch{X_0.}{X_0}
\end{align*}

Factoring each summand as the difference of squares, we get:
\begin{align*}
X_t &= X_0+\sum_{i=1}^n
\Biggl[2\sqrt{X_0} + \frac{1}{2} \Biggl(\int
_0^{t_{i}} \biggl(\frac{k}{Y_s}-aY_s
\biggr)ds
\\
&\quad+\int_0^{t_{i-1}} \biggl(\frac{k}{Y_s}-aY_s
\biggr)ds \Biggr)+\frac
{\sigma}{2} \bigl(B_{t_i}^H +
B_{t_{i-1}}^H \bigr) \Biggr]
\\
&\quad\times \Biggl[\frac{1}{2}\int_{t_{i-1}}^{t_{i}}
\biggl(\frac
{k}{Y_s}-aY_s \biggr)ds + \frac{\sigma}{2}
\bigl(B_{t_i}^H - B_{t_{i-1}}^H \bigr)
\Biggr].
\end{align*}

Expanding the brackets in the last expression, we obtain:
\begin{align}
X_t &= X_0+\sum_{i=1}^n \sqrt{X_0}\int_{t_{i-1}}^{t_{i}} \biggl(\frac{k}{Y_s}-aY_s\biggr)ds\nonumber\\
&\quad +\frac{1}{4}\sum_{i=1}^n \Biggl(\int_0^{t_{i}} \biggl(\frac{k}{Y_s}-aY_s\biggr)ds+\int_0^{t_{i-1}} \biggl(\frac{k}{Y_s}-aY_s \biggr)ds \Biggr)\nonumber\\
&\quad \times\int_{t_{i-1}}^{t_{i}} \biggl(\frac{k}{Y_s}-aY_s\biggr)ds+\frac{\sigma}{4}\sum_{i=1}^n\bigl(B_{t_i}^H + B_{t_{i-1}}^H \bigr)\int_{t_{i-1}}^{t_{i}} \biggl(\frac{k}{Y_s}-aY_s\biggr)ds\nonumber\\
&\quad + \sigma\sqrt{X_0}\sum_{i=1}^n\bigl(B_{t_i}^H - B_{t_{i-1}}^H \bigr)+\frac{\sigma^2}{4}\sum_{i=1}^n\bigl(B_{t_i}^H - B_{t_{i-1}}^H \bigr)\bigl(B_{t_i}^H + B_{t_{i-1}}^H \bigr)\nonumber\\
&\quad + \frac{\sigma}{4}\sum_{i=1}^n \Biggl(\int_0^{t_{i}} \biggl(\frac{k}{Y_s}-aY_s\biggr)ds+\int_0^{t_{i-1}} \biggl(\frac{k}{Y_s}-aY_s \biggr)ds \Biggr) \bigl(B_{t_i}^H- B_{t_{i-1}}^H \bigr).\label{integral sums 2}
\end{align}

Let the mesh $\Delta t$ of the partition tend to zero. The first three summands
\begin{align}
&\sum_{i=1}^n\sqrt{X_0}\int_{t_{i-1}}^{t_{i}} \biggl(\frac{k}{Y_s}-aY_s \biggr)ds \nonumber\\
&\qquad +\frac{1}{4}\sum_{i=1}^n \Biggl(\int_0^{t_{i}} \biggl(\frac{k}{Y_s}-aY_s\biggr)ds+\int_0^{t_{i-1}} \biggl(\frac{k}{Y_s}-aY_s \biggr)ds \Biggr)\nonumber\\
&\qquad \times\int_{t_{i-1}}^{t_{i}} \biggl(\frac{k}{Y_s}-aY_s\biggr)ds+\frac{\sigma}{4}\sum_{i=1}^n\bigl(B_{t_i}^H + B_{t_{i-1}}^H \bigr)\int_{t_{i-1}}^{t_{i}} \biggl(\frac{k}{Y_s}-aY_s\biggr)ds\nonumber\\
&\quad \rightarrow\int_0^t\biggl(\frac{k}{Y_s}-aY_s \biggr) \Biggl(\sqrt{X_0}+ \frac{1}{2}\int_0^s \biggl(\frac{k}{Y_u}-aY_u \biggr)du+\frac{\sigma}{2}B_s^H\Biggr)ds\nonumber\\
&\quad =\int_0^t \bigl(k-aY_s^2\bigr)ds=\int_0^t (k-aX_s )ds,\quad \Delta t\rightarrow0, \label{lebesgue sums 2}
\end{align}
and the last three summands
\begin{align}
&\sigma\sqrt{X_0}\sum_{i=1}^n \bigl(B_{t_i}^H- B_{t_{i-1}}^H \bigr)+\frac{\sigma^2}{4}\sum_{i=1}^n \bigl(B_{t_i}^H -B_{t_{i-1}}^H \bigr) \bigl(B_{t_i}^H +B_{t_{i-1}}^H \bigr)\nonumber\\
&\qquad + \frac{\sigma}{4}\sum_{i=1}^n \Biggl(\int_0^{t_{i}} \biggl(\frac{k}{Y_s}-aY_s\biggr)ds+\int_0^{t_{i-1}} \biggl(\frac{k}{Y_s}-aY_s \biggr)ds \Biggr) \bigl(B_{t_i}^H- B_{t_{i-1}}^H \bigr) \nonumber\\
&\quad \rightarrow\sigma\int_0^t \Biggl(\sqrt{X_0} + \frac{1}{2}\int_0^s\biggl(\frac{k}{Y_u}-aY_u \biggr)du+\frac{\sigma}{2}B_s^H\Biggr)\circ dB_s^H \nonumber\\
&\quad =\sigma\int_0^t Y_s \circ dB_s^H = \sigma\int_0^t\sqrt{X_s}\circ dB_s^H,\quad\Delta t \rightarrow0.\label{stratonovich sums 2}
\end{align}

Note that the left-hand side of \eqref{integral sums 2} does not depend
on the partition and the limit in \eqref{lebesgue sums 2} exists as the
pathwise Riemann integral, therefore the corresponding pathwise
Stratonovich integral exists and the passage to the limit in \eqref
{stratonovich sums 2} is correct.

Thus, the fractional Cox--Ingersoll--Ross process, introduced
in Definition~\ref{generfCIR}, satisfies the SDE of the form
\begin{equation}
\begin{gathered}\label{stratonovich equation 2} X_t = X_0 +
\int_0^t (k-aX_s )ds+\sigma\int
_0^t\sqrt{X_s}\circ
dB_s^H, \end{gathered} %
\end{equation}
where $\int_0^t\sqrt{X_s}\circ dB_s^H$ is the pathwise Stratonovich integral.
\end{proof}

\begin{remark}
In the case of $k=0$, the process \eqref{new sde} is the fractional
Ornstein--Uhlenbeck process and the definition coincides with the one
given in \cite{MPRY}.
\end{remark}

\section{Hitting zero by the fractional CIR process with positive
``mean'' and $\mathbf{H>1/2}$}\label{sec3}

The next natural question regarding the fractional CIR process is
finiteness of its zero hitting time moment. It is obvious that it
coincides with the respective moment of the process $\{Y_t,t\ge0\}$,
defined by the equation \eqref{new sde}.

Before formulating the main result of the section let us give a
well-known property of trajectories of fractional Brownian motion (see,
for example, \cite{MishurafBm}).

\begin{prop}\label{fBm property}
Let $\{B_t^H,t\ge0\}$ be a fractional Brownian motion with the Hurst
index $H$. Then, $\exists\varOmega'\subset\varOmega$, $\mathbb P \{
\varOmega
'\}
=1$, such that $\forall\omega\in\varOmega'$, $\forall T>0$, $\forall
\delta>0$, $\forall0 \le s \le t \le T$ $\exists C=C(T,\omega,
\delta
)\in\mathbb{R}$:
\begin{equation*}
\bigl\llvert B_t^H-B_s^H \bigr
\rrvert\le C\llvert t-s\rrvert^{H-\delta}.
\end{equation*}
\end{prop}

\begin{thm}\label{main result}
Let $k>0, H>1/2$. Then the process $\{Y_t,t\ge0\}$, defined by the
equation \eqref{new sde}, is strictly positive a.s.
\end{thm}

\begin{proof}
The proof is by contradiction.

Let \querymark{Q1}$\varOmega'$ be the same as in Proposition~\ref{fBm property}. First,
assume that $a>0$ and let for some $\omega\in\varOmega'$, $\tau(\omega
)=\inf
\{t>0:X_t=0\}=\inf\{t>0:Y_t=0\}<\infty$.

For all $\varepsilon\in(0,\min(Y_0,\sqrt{\frac{k}{a}}))$ (the condition
$\varepsilon< \sqrt{\frac{k}{a}}$ provides the inequality $\frac
{k}{\varepsilon}-a\varepsilon>0$) let us introduce the last moment of
hitting the level of $\varepsilon$ before the first zero reaching:
\begin{equation*}
\tau_\varepsilon:=\sup \bigl\{t\in(0,\tau):Y_t=\varepsilon \bigr
\}.
\end{equation*}

Consider $\delta>0$ such that the inequality $H-\delta>1/2
\Leftrightarrow1+\delta-H<1/2$ holds. According to the definitions of
$\tau, \tau_\varepsilon$ and $Y$, the following equality is true:
\begin{equation*}
\begin{gathered} -\varepsilon=Y_\tau-Y_{\tau_\varepsilon}=
\frac{1}{2}\int_{\tau
_\varepsilon}^\tau \biggl(
\frac{k}{Y_s}-aY_s \biggr)ds+\frac{\sigma
}{2}
\bigl(B_\tau^H-B_{\tau_\varepsilon}^H \bigr).
\end{gathered} %
\end{equation*}

The process $Y_s\in(0,\varepsilon)$ on the interval $(\tau
_\varepsilon
,\tau)$, hence $\forall s\in(\tau_\varepsilon,\tau)$:
\begin{equation}
\label{positive a} %
\begin{gathered} \frac{k}{Y_s}-aY_s\ge
\frac{k}{\varepsilon}-a\varepsilon. \end{gathered} %
\end{equation}

From this and Proposition~\ref{fBm property}, it follows that $\exists
C=C(\tau(\omega),\omega,\delta)$:
\begin{equation*}
\begin{gathered} \frac{\sigma}{2}C\llvert\tau-\tau_\varepsilon
\rrvert^{H-\delta}\ge\frac
{\sigma
}{2} \llvert B_\tau^H-B_{\tau_\varepsilon}^H
\rvert\ge\frac
{1}{2} \biggl(\frac{k}{\varepsilon}-a\varepsilon \biggr) (
\tau-\tau_\varepsilon)+\varepsilon. \end{gathered} %
\end{equation*}

It is clear that there exists $\tilde{\varepsilon}>0$ such that
$\forall\varepsilon<\tilde{\varepsilon}$: $\frac{k}{\varepsilon
}-a\varepsilon>\frac{k}{2\varepsilon}$. Then, by choosing an arbitrary
$\varepsilon<\tilde{\varepsilon}$, we have:
\begin{equation}
\label{inequality} %
\begin{gathered} \frac{\sigma}{2}C\llvert\tau-
\tau_\varepsilon\rrvert^{H-\delta}\ge\frac
{k}{4\varepsilon} (\tau-
\tau_\varepsilon)+\varepsilon. \end{gathered} %
\end{equation}

For $x\ge0$ consider the function
\begin{equation}
\begin{gathered} \label{key function} F_\varepsilon(x)=
\frac{k}{4\varepsilon}x-\frac{\sigma
}{2}Cx^{H-\delta
}+\varepsilon. \end{gathered}
\end{equation}

Let us show that there exists $\varepsilon^*\in(0,\tilde{\varepsilon})$
such that, for all $\varepsilon<\varepsilon^*$ and for all $x\ge0$,
$F_\varepsilon(x)>0$. It is easy to check that $F_\varepsilon
(0)=\varepsilon>0$ and $F_\varepsilon$ is convex on $\mathbb
{R}^+\backslash\{0\}$ (its second derivative is strictly positive on
this set), so it is enough to examine the sign of the function in its
critical points.
\begin{align*}
F'(\tilde{x})&=\frac{k}{4\varepsilon}-\frac{\sigma(H-\delta)}{2}C\tilde{x}^{H-\delta-1}=0 \\
\implies\quad \tilde{x}&= \biggl(\frac{k}{2\sigma\varepsilon C(H-\delta)} \biggr)^{1/(H-\delta-1)}\\
&= \biggl(\frac{2\sigma C (H-\delta)}{k} \biggr)^{1/(1+\delta-H)}\varepsilon^{1/ (1+\delta-H )}.
\end{align*}
After some calculations we get
\begin{equation*}
\begin{gathered} F(\tilde{x})=\frac{1}{2} \biggl(
\frac{2(H-\delta)}{k} \biggr)^{\frac
{H-\delta}{1+\delta-H}} (\sigma C )^{\frac{1}{1+\delta
-H}}(H-\delta-1)
\varepsilon^{\frac{H-\delta}{1+\delta
-H}}+\varepsilon. \end{gathered} %
\end{equation*}

From the choice of $\delta$ it follows that $\frac{H-\delta
}{1+\delta
-H}>1$, hence $\forall K\in\mathbb{R}\quad\exists\varepsilon^*>0$:
\begin{equation}
\begin{gathered} \label{key inequality} \varepsilon-K\varepsilon^{\frac
{H-\delta}{1+\delta-H}}>0,
\quad\forall\varepsilon<\varepsilon^*. \end{gathered} %
\end{equation}

Choosing the corresponding $\varepsilon^*$ for
\begin{equation*}
\begin{gathered} K:=-\frac{1}{2} \biggl(\frac{2(H-\delta)}{k}
\biggr)^{\frac{H-\delta
}{1+\delta-H}} (\sigma C )^{\frac{1}{1+\delta
-H}}
\xch{(H-\delta-1),}{(H-\delta-1)}
\end{gathered}
\end{equation*}
and choosing an arbitrary $\varepsilon<\min\{\tilde{\varepsilon
},\varepsilon^*\}$ we obtain that
\begin{equation*}
F_\varepsilon(x)>0 \quad\forall x>0.
\end{equation*}
However, from \eqref{inequality} it follows that
\begin{equation*}
F_\varepsilon(\tau-\tau_\varepsilon)\le0.
\end{equation*}

The contradiction obtained proves the theorem for $a>0$. If $a\le0$,
instead of \eqref{positive a} the following bound can be used:
\begin{equation}
\label{negative a} %
\frac{k}{Y_s}-aY_s\ge \frac{k}{\varepsilon}.\newqedhere
\end{equation}
\end{proof}

\section{Hitting zero by the fractional CIR process in the case of
$H<1/2$}\label{sec4}

The condition of $H>1/2$ is essential for Theorem~\ref{main result}, as
if $H<1/2$, the condition \eqref{key inequality} does not hold.
However, it is possible to obtain another result concerning zero
hitting by the fractional CIR process in the case of $H<1/2$.

Let $\{B_t^H, t\ge0\}$ be the fractional Brownian motion with $H<1/2$
and let $a\in\mathbb{R}$, $\sigma>0$ be fixed. Consider the set of processes
\begin{equation}
\label{set of Ys} \mathbb{Y}:= \bigl\{Y^{(k)}= \bigl\{Y_t^{(k)},
t\ge0 \bigr\}, k>0 \bigr\},
\end{equation}
each element of which starts from the same level $Y_0>0$, satisfies the
SDE of the form \eqref{new sde} before hitting zero and remains in zero
after that moment:
\begin{equation*}
Y^{(k)}_t(\omega)= %
\begin{cases}
Y_0+\frac{1}{2}\int_0^t (\frac{k}{Y^{(k)}_s(\omega)}-aY^{(k)}_s(\omega) )ds+\frac{\sigma}{2}dB_t^H(\omega), &\text{if $t<\tau^{(k)}(\omega)$}, \\
0, &\text{if $t\ge\tau^{(k)}(\omega)$},
\end{cases}
\end{equation*}
where $\tau^{(k)}:=\inf\{t\ge0 \mid Y^{(k)}_t=0\}$.

\begin{lemma}\label{comparison thm}
Let $k_1<k_2$. Then $\forall\omega\in\varOmega, \forall t\ge0$:
\begin{enumerate}
\item[(i)] $\tau^{(k_1)}(\omega)\le\tau^{(k_2)}(\omega)$;
\item[(ii)] $Y^{(k_1)}_t(\omega)\le Y^{(k_2)}_t(\omega)$, and the
inequality is strict for $t\in(0,\tau^{(k_2)}(\omega))$.
\end{enumerate}
\end{lemma}

\begin{remark}
This lemma holds for an arbitrary Hurst index $H\in(0,1)$.
\end{remark}

\begin{proof}
Let $\omega\in\varOmega$ be fixed (we will omit $\omega$ in brackets in
further formulas). Consider the function $\delta$ on the interval $[0,
\min\{\tau^{(k_1)},\tau^{(k_2)}\})$, such that $\delta
(t)=Y^{(k_2)}_t-Y^{(k_1)}_t$. It is obvious that $\delta$ is
differentiable, $\delta(0)=0$ and
\begin{equation*}
\delta_+'(0)=\frac{1}{2} \biggl(\frac{k_2}{Y_0}-aY_0
\biggr)-\frac
{1}{2} \biggl(\frac{k_1}{Y_0}-aY_0 \biggr)=
\frac{k_2-k_1}{2Y_0}>0.
\end{equation*}

As $\delta(t)=\delta_+'(0)t+o(t)$, $t\rightarrow0+$, it is easy to see
that there exists the maximal interval $(0,t^*)\subset(0, \min\{\tau
^{(k_1)},\tau^{(k_2)}\})$ such that $\delta(t)>0$ for all $t\in
(0,t^*)$. It is also clear that
\begin{equation*}
t^*=\sup \bigl\{t\in \bigl(0, \min \bigl\{\tau^{(k_1)},
\tau^{(k_2)} \bigr\} \bigr) \mid\forall s\in(0,t): \delta(s)>0 \bigr\}.
\end{equation*}

Assume that $t^*< \min\{\tau^{(k_1)},\tau^{(k_2)}\}$.
According to the definition of $t^*$, $\delta(t^*)=0$. Hence
$Y^{(k_2)}_{t^*}=Y^{(k_1)}_{t^*}=Y^*>0$ and
\begin{equation*}
\delta' \bigl(t^* \bigr)=\frac{k_2-k_1}{2Y^*}>0.
\end{equation*}

As $\delta(t)=\delta'(t^*)(t-t^*)+o(t-t^*)$, $t\rightarrow t^*$, there
exists $\varepsilon>0$ such that $\delta(t)<0$ for all $t\in
(t^*-\varepsilon, t^*)$, that contradicts the definition of $t^*$.

Therefore, $\forall t\in(0, \min\{\tau^{(k_1)},\tau^{(k_2)}\})$:
\begin{equation}
\label{comp ineq} Y^{(k_2)}_t>Y^{(k_1)}_t.
\end{equation}

Now it is easy to show that $(i)$ holds: indeed, if $\tau^{(k_1)}>\tau
^{(k_2)}$, then
\[
0=Y^{(k_2)}_{\tau^{(k_2)}}<Y^{(k_1)}_{\tau^{(k_2)}}.
\]

This means that $\exists t_*<\tau^{(k_2)}$ such that
$Y^{(k_2)}_t<Y^{(k_1)}_t$ for all $t\in(t_*,\tau^{(k_2)})$, which
contradicts \eqref{comp ineq}.

Finally, as $Y^{(k_2)}_t>Y^{(k_1)}_t$ for all $t\in(0, \tau^{(k_1)})$,
$Y^{(k_2)}_t>Y^{(k_1)}_t=0$ for all $t\in[\tau^{(k_1)}, \tau^{(k_2)})$
and $Y^{(k_2)}_t=Y^{(k_1)}_t=0$ for all $t\ge\tau^{(k_2)}$, $(ii)$
also holds.
\end{proof}

Now let us move to the main result of the section.
\begin{thm}\label{less main result} For all $T>0$:
\begin{equation}
\mathbb{P} \bigl(\tau^{(k)}>T \bigr)\rightarrow1,\quad k\rightarrow
\infty.
\end{equation}
\end{thm}
\begin{proof}
The proof is by contradiction.

Assume that $\exists T^*>0$, $\exists\{k_n, n\ge1\}, k_n\uparrow
\infty
$ as $n\rightarrow\infty$ such that:
\begin{equation*}
\mathbb{P} \bigl(\tau^{(k_n)}\le T^* \bigr)\rightarrow\alpha>0, \quad n
\rightarrow\infty.
\end{equation*}

Let us consider the case of $a>0$. Let $\varOmega'$ be from Proposition
\ref{fBm property}, and for all $\varepsilon\in(0,\min(Y_0, 1,
\sqrt
{\frac{k_1}{2a}}))$ denote $\tau^{(k_n)}_\varepsilon:=\sup\{t\in
(0,\tau
):Y^{(k_n)}_t=\varepsilon\}$ and
\begin{equation*}
\begin{gathered} D^{(k_n)}_{T^*}:= \bigl\{\omega\in
\varOmega'\mid\tau^{(k_n)}\le T^* \bigr\}. \end{gathered}
\end{equation*}

According to Lemma~\ref{comparison thm}, $\forall n\ge1:
D^{(k_{n+1})}_{T^*}\subset D^{(k_{n})}_{T^*}$, so
\begin{equation*}
\mathbb{P} \biggl(\bigcap_{n\ge1}D^{(k_n)}_{T^*}
\biggr)=\lim_{n\rightarrow
\infty}\mathbb{P} \bigl(D^{(k_n)}_{T^*}
\bigr)=\alpha>0.
\end{equation*}

Just like in Theorem~\ref{main result}, $\forall n\ge1$, $\forall
\omega
\in D^{(k_n)}_{T^*}$:
\begin{equation*}
\begin{gathered} -\varepsilon=Y^{(k_n)}_{\tau^{(k_n)}}-Y^{(k_n)}_{\tau
^{(k_n)}_\varepsilon}=
\frac{1}{2}\int_{\tau^{(k_n)}_\varepsilon
}^{\tau
^{(k_n)}} \biggl(
\frac{k_n}{Y^{(k_n)}_s}-aY^{(k_n)}_s \biggr)ds+\frac
{\sigma}{2}
\bigl(B_{\tau^{(k_n)}}^H-B_{{\tau^{(k_n)}_\varepsilon
}}^H \bigr).
\end{gathered} %
\end{equation*}

The process $Y^{(k_n)}_s\in(0,\varepsilon)$ on the interval $(\tau
^{(k_n)}_\varepsilon,\tau^{(k_n)})$, hence
\begin{equation}
\label{positive a2} %
\begin{gathered} \frac{k_n}{Y^{(k_n)}_s}-aY^{(k_n)}_s
\ge\frac{k_n}{\varepsilon
}-a\varepsilon, \quad\forall s\in \bigl(\tau^{(k_n)}_\varepsilon,
\tau^{(k_n)} \bigr). \end{gathered} %
\end{equation}

Let $\delta>0$ satisfy the condition $0<H-\delta<1/2$.

According to Proposition~\ref{fBm property}, $\exists C_\omega
=C(T^*,\omega,\delta)$, $\forall0<s<t<T^*$:
\begin{equation*}
\bigl\llvert B_t^H-B_s^H \bigr
\rrvert\le C_\omega\llvert t-s\rrvert^{H-\delta}.
\end{equation*}

As $\varepsilon< \sqrt{\frac{k_1}{2a}}$, the following inequality is true:
\begin{equation*}
\begin{gathered} \frac{k_n}{\varepsilon}-a\varepsilon>\frac
{k_n}{2\varepsilon}
\quad\forall n\ge0, \end{gathered} %
\end{equation*}
so just like in the proof of Theorem~\ref{main result} we can obtain
that $\forall n\ge1$, $\forall\omega\in D^{(k_n)}_{T^*}$:
\begin{equation}
\label{on the one hand} \frac{\sigma}{2}C_\omega \bigl(\tau^{(k_n)}-
\tau^{(k_n)}_\varepsilon \bigr)^{H-\delta}\ge\frac{k_n}{4\varepsilon}
\bigl(\tau^{(k_n)}-\tau^{(k_n)}_\varepsilon \bigr)+\varepsilon.
\end{equation}
According to \eqref{on the one hand}, $\forall n\ge1$, $\forall
\omega
\in\bigcap_{n\ge0}D^{(k_n)}_{T^*}$:
\begin{equation}
\label{contradiction} \frac{\sigma}{2}C_\omega \bigl(\tau^{(k_n)}-
\tau^{(k_n)}_\varepsilon \bigr)^{H-\delta}>\varepsilon.
\end{equation}
However, it is easy to see from \eqref{on the one hand} that
\begingroup
\abovedisplayskip=4pt
\belowdisplayskip=4pt
\begin{equation}
\begin{gathered}\label{on the other hand} \frac{\sigma}{2}C_\omega
\bigl(\tau^{(k_n)}-\tau^{(k_n)}_\varepsilon \bigr)^{H-\delta}>
\frac{k_n}{4\varepsilon} \bigl(\tau^{(k_n)}-\tau^{(k_n)}_\varepsilon
\bigr). \end{gathered} %
\end{equation}

Let us transform \eqref{on the other hand}:
\begin{align*}
\frac{\sigma}{2}C_\omega&>\frac{k_n}{4\varepsilon}\bigl(\tau^{(k_n)}-\tau^{(k_n)}_\varepsilon\bigr)^{1-H+\delta},\\
\frac{2\sigma\varepsilon}{k_n}C_\omega&> \bigl(\tau^{(k_n)}-\tau^{(k_n)}_\varepsilon \bigr)^{1-H+\delta},\\
\biggl(\frac{2\sigma\varepsilon}{k_n}C_\omega \biggr)^{\frac{1}{1-H+\delta}}&>\tau^{(k_n)}-\tau^{(k_n)}_\varepsilon,
\end{align*}
hence
\begin{align*}
\frac{\sigma}{2}C_\omega \bigl(\tau^{(k_n)}-\tau^{(k_n)}_\varepsilon \bigr)^{H-\delta}&<\frac{\sigma}{2}C_\omega \biggl(\frac{2\sigma\varepsilon}{k_n}C_\omega\biggr)^{\frac{H-\delta}{1-H+\delta}}\\
&= \bigl(2^{\frac{2H-2\delta-1}{1-H+\delta}}\sigma^{\frac{1}{1-H+\delta}} \bigr)k_n^{-\frac{H-\delta}{1-H+\delta}}C_\omega^{\frac{1}{1-H+\delta}}\varepsilon^{\frac{H-\delta}{1-H+\delta}}\\
&=\tilde{C}k_n^{-\frac{H-\delta}{1-H+\delta}}C_\omega^{\frac{1}{1-H+\delta}}\varepsilon^{\frac{H-\delta}{1-H+\delta}}.
\end{align*}

According to \cite{Nual-Rasc}, \querymark{Q2}$\mathbb{E} (\llvert
C_\omega\rrvert^p)<\infty$ for all $p\in[1,\infty)$, so $C_\omega$ is finite a.s.

Therefore, as $\mathbb{P} (\bigcap_{n\ge0}D^{(k_n)}_{T^*}
)=\alpha>0$, $\exists M>0$, $\exists E \subset\bigcap_{n\ge
0}D^{(k_n)}_{T^*}$, $\mathbb{P}(E)>0$ such that $\forall\omega\in E$:
\begin{equation*}
C_\omega< M.
\end{equation*}

Hence, as $\varepsilon<1$,
\begin{align*}
\frac{\sigma}{2}C_\omega \bigl(\tau^{(k_n)}-\tau^{(k_n)}_\varepsilon \bigr)^{H-\delta}&<\tilde{C}k_n^{-\frac{H-\delta}{1-H+\delta}}C_\omega^{\frac{1}{1-H+\delta}}\varepsilon^{\frac{H-\delta}{1-H+\delta}}\\
&<\tilde{C}k_n^{-\frac{H-\delta}{1-H+\delta}}M^{\frac{1}{1-H+\delta}}<\varepsilon,
\end{align*}
if $k_n> (\frac{\tilde{C}M^{\frac{1}{1-H+\delta}}}{\varepsilon
} )^{\frac{1-H+\delta}{H-\delta}}$, which contradicts \eqref
{contradiction}.

If $a<0$, the following inequality can be used instead of \eqref
{positive a2}:
\begin{equation*}
\frac{k_n}{Y^{(k_n)}_s}-aY^{(k_n)}_s\ge\frac{k_n}{\varepsilon}.\qedhere
\end{equation*}
\end{proof}
\endgroup

\appendix
\section{Appendix: Simulations of the fractional Cox--Ingersoll--Ross
process}\label{append}

Theorems~\ref{main result} and~\ref{less main result} can be
illustrated by numerical simulations.

10000 sample paths of the fractional Cox--Ingersoll--Ross
process were simulated on the interval $[0,10]$ as the square of the
process $Y$ defined in \eqref{new sde}. The Euler approximation of $Y$
was used until the first zero hitting by the latter with the mesh of
the partition of $\Delta t=0.001$:
\begin{align*}
Y_{t_n}&= %
\begin{cases}
Y_{t_{n-1}}+\frac{1}{2} (\frac
{k}{Y_{t_{n-1}}}-aY_{t_{n-1}}
)\Delta t+\frac{\sigma}{2}\Delta B_{t_n}^H, &\text{if $Y_{t_{n-1}}>0$},
\\
0, &\text{if $Y_{t_{n-1}}\le0$},
\end{cases}
\\
X_{t_n}&=Y^2_{t_n}.
\end{align*}

There were no zero hitting for 10000 trajectories simulated for four cases
that satisfy the conditions of Theorem~\ref{main result} (see Fig.~\ref{fig1}, \ref{fig2}; the amount of trajectories on these and further
figures is reduced in order to make them more convenient for the reader).\looseness=1

\begin{figure}[h]
\centering
\begin{minipage}[b]{0.49\textwidth}
\includegraphics{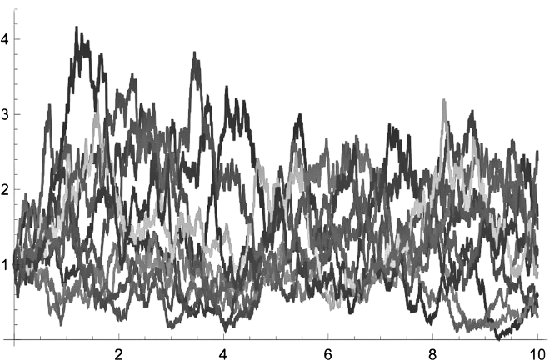}
\caption{Case of $a=1$, $ k=1$, $\sigma=1$, $H=0.6$, $X_0=1$}\label{fig1}
\end{minipage}
\figend
\begin{minipage}[b]{0.49\textwidth}
\includegraphics{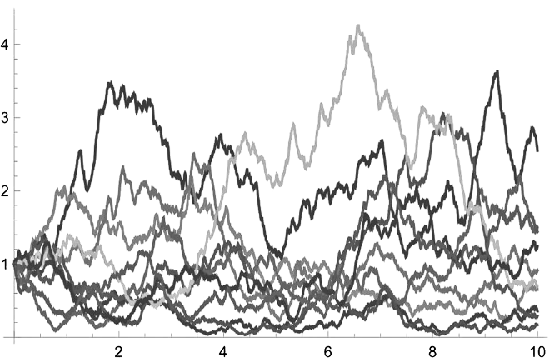}
\caption{Case of $a=1$, $ k=1$, $\sigma=1$, $H=0.8$, $X_0=1$}\label{fig2}
\end{minipage}
\end{figure}

\begin{figure}[b!]
\centering
\begin{minipage}[b]{0.49\textwidth}
\includegraphics{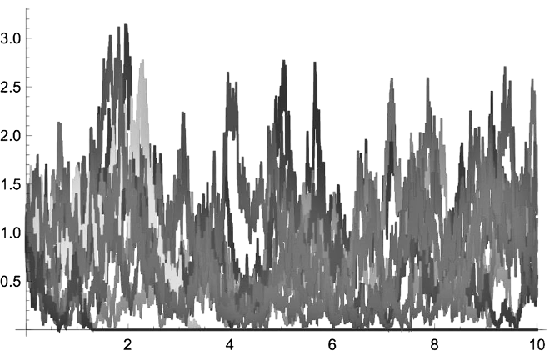}
\caption{Case of $a=1$, $k=0.5$, $\sigma=1$, $H=0.4$, $X_0=1$. 17\% of
sample paths hit zero}\label{fig5}
\end{minipage}
\figend
\begin{minipage}[b]{0.49\textwidth}
\includegraphics{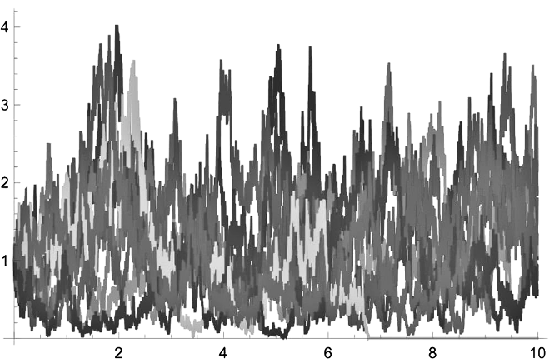}
\caption{Case of $a=1$, $ k=1$, $\sigma=1$, $H=0.4$, $X_0=1$. Less than
1\% of
sample paths hit zero}\label{fig6}
\end{minipage}
\end{figure}

\begin{figure}[b!]
\centering
\begin{minipage}[b]{0.49\textwidth}
\includegraphics{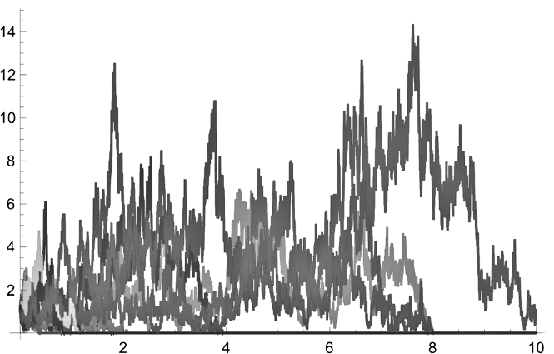}
\caption{Case of $a=1$, $k=1$, $\sigma=2$, $H=0.4$, $X_0=1$. 86\% of
sample paths hit zero}\label{fig7}
\end{minipage}
\figend
\begin{minipage}[b]{0.49\textwidth}
\includegraphics{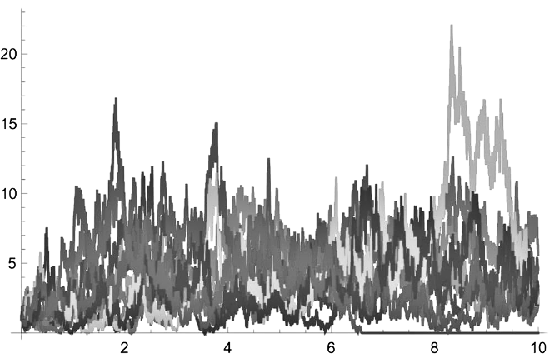}
\caption{Case of $a=1$, $ k=3$, $\sigma=2$, $H=0.4$, $X_0=1$. Less than
1\% of
sample paths hit zero}\label{fig8}
\end{minipage}
\end{figure}

\begin{figure}[h!]
\centering
\begin{minipage}[b]{0.49\textwidth}
\includegraphics{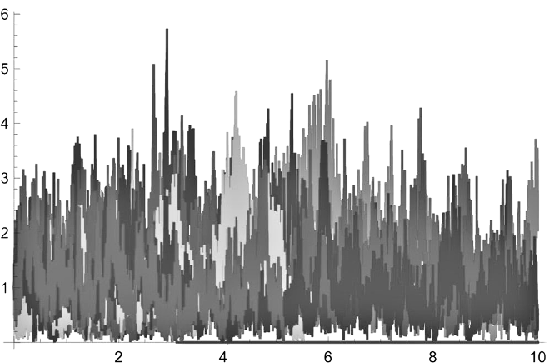}
\caption{Case of $a=1$, $k=1$, $\sigma=1$, $H=0.2$, $X_0=1$. 43\% of
sample paths hit zero}\label{fig9}
\end{minipage}
\figend
\begin{minipage}[b]{0.49\textwidth}
\includegraphics{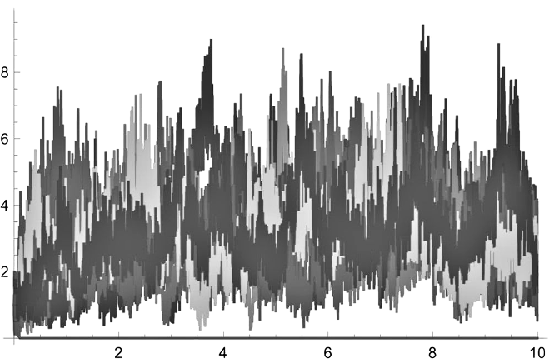}
\caption{Case of $a=1$, $ k=3$, $\sigma=1$, $H=0.2$, $X_0=1$. Less than
1\% of
sample paths hit zero}\label{fig10}
\end{minipage}
\end{figure}

However, the \xch{behavior}{behaviour} of the fractional Cox--Ingersoll--Ross
process is not completely clear for the situation of $k>0, H<1/2$.
Simulations for different parameters $k, H, \sigma$ (see Figures~\ref
{fig5}--\ref{fig10}) show that in this case the process may hit zero
with positive probability, however, as stated in Theorem~\ref{less main
result}, as $k$ gets bigger, the amount of trajectories that hit zero
tends to zero. Moreover, it seems that the less $H$ and the bigger
$\sigma$ are, the bigger is the probability of reaching zero.

%\bibliography{bib/biblio}

%\begin{appendix}
%\end{appendix}

%\begin{acknowledgement}%[title={Acknowledgments}]
%\end{acknowledgement}

%\begin{funding}
%\gsponsor[id=,sponsor-id=]{}
%\gnumber[refid=]{}
%\end{funding}

%
%\begin{thebibliography}{}
%\end{thebibliography}
\end{document}